\documentclass{aic}

\usepackage{amsthm, amsmath, amsfonts, cite, enumitem}
\usepackage{graphicx}
\usepackage{caption}
\usepackage{titlesec}
\graphicspath{ {./Paperfigures/} }
\usepackage{epsfig}
\usepackage{caption}  
\usepackage{capt-of}
\usepackage{tikz}
\usepackage{enumitem}
\usepackage{multicol}
\usepackage{thmtools}
\usepackage{thm-restate}
\usepackage{cleveref}
\newtheorem{theorem}{Theorem}

\newtheorem{corollary}[theorem]{Corollary}

\aicAUTHORdetails{%
  title = {Box and Segment Intersection Graphs with Large Girth and Chromatic Number}, 
  author = {James Davies},
  plaintextauthor = {James Davies},
}   

\aicEDITORdetails{%
   year={2021},
   number={7},
   received={1 December 2020},   
   published={12 July 2021},  
   doi={10.19086/aic.25431},      
}   

\begin{document}

\begin{frontmatter}[classification=text]


\author[jd]{James Davies}

\begin{abstract}
We prove that there are intersection graphs of axis-aligned boxes in $\mathbb{R}^3$ and intersection graphs of straight lines in $\mathbb{R}^3$ that have arbitrarily large girth and chromatic number.
\end{abstract}
\end{frontmatter}

\section{Introduction}\label{section1}

Erd{\H{o}}s~\cite{erdos1959graph} proved that there exist graphs with arbitrarily large girth and chromatic number. We prove that there exist such graphs that can be realised geometrically as intersection graphs of axis-aligned boxes in $\mathbb{R}^3$ and of straight lines in $\mathbb{R}^3$. To do so, we make use of a surprising connection to arithmetic Ramsey theory. These are the first non-trivial examples of geometric intersection graphs with arbitrarily large girth and chromatic number.

\begin{theorem}\label{main}
	There are intersection graphs of axis-aligned boxes in $\mathbb{R}^3$ with arbitrarily large girth and chromatic number.
\end{theorem}

\begin{theorem}\label{main line}
	There are intersection graphs of straight lines in $\mathbb{R}^3$ with arbitrarily large girth and chromatic number.
\end{theorem}

In 1965, Burling~\cite{burling1965coloring} proved that there are axis-aligned boxes in $\mathbb{R}^3$ whose intersection graphs are triangle-free and have arbitrarily large chromatic number.
In light of Burling's construction, the problem of whether intersection graphs of axis-aligned boxes in $\mathbb{R}^3$ with large girth have bounded chromatic number was raised by Kostochka and Perepelitsa~\cite{kostochka2000colouring}. Later Kostochka~\cite{kostochka2004coloring} further speculated, for all positive integers $d$, that intersection graphs of axis-aligned boxes in $\mathbb{R}^d$ with girth at least 5 should have bounded chromatic number. Theorem~\ref{main} extends Burling's result and shows that this is not the case even with any given larger girth in $\mathbb{R}^3$. We prove Theorem~\ref{main} in the more restricted setting of ``grounded square boxes'' where the boxes intersect a given plane and their top and bottom faces are squares (see Theorem~\ref{grounded boxes}).

In the plane, the situation is different. In 1948, Bielecki~\cite{colouring1948} asked if triangle-free intersection graphs of axis-aligned rectangles have bounded chromatic number. This was answered positively by Asplund and Grünbaum~\cite{grunbaum1960coloring}, who proved the more general statement that axis-aligned rectangle intersection graphs with bounded clique number have bounded chromatic number, i.e., they are $\chi$-bounded. The best known $\chi$-bounding function for such graphs is $O(\omega \log \omega)$~\cite{chalermsook2020coloring}.

The problem of whether intersection graphs of lines in $\mathbb{R}^3$ are $\chi$-bounded had been raised explicitly by Pach, Tardos, and T{\'o}th~\cite{pach2017}, but had been circulating in the community a few years prior.
Norin~\cite{norin} (see~\cite{pach2020}) answered this in the negative by showing that  double shift graphs are line intersection graphs. Since Norin's construction had only been communicated privately, we present it in Section~\ref{section4} before extending the result by proving Theorem~\ref{main line}.

Intersection graphs of straight (infinite) lines are a subclass of the intersection graphs of straight line segments. Thus Theorem~\ref{main line} also applies to the class of intersection graphs of segments in $\mathbb{R}^3$.
The significance of Burling's construction was highlighted when it was used much later to show that there are triangle-free intersection graphs of segments in the plane with arbitrarily large chromatic number~\cite{pawlik2014triangle}.
This disproved Scott's~\cite{scott1997trees} conjecture that graphs not containing an induced subdivision of a given graph $H$ are $\chi$-bounded.
In light of Theorem~\ref{main}, segment intersection graphs are perhaps a more interesting setting for Burling graphs, indeed it is even conjectured~\cite{chudnovsky2016induced} that Burling graphs might be the only obstruction to intersection graphs of segments in the plane being $\chi$-bounded.

Kostochka and Ne\v{s}et\v{r}il~\cite{kostochka1998coloring} proved that if the intersection graph of $n$ segments in the plane has girth at least 5, then the graph has only $O(n)$ edges, consequently such graphs have bounded chromatic number. Fox and Pach~\cite{fox2008separator} asked if this result could be extended to segments in $\mathbb{R}^3$; Theorem~\ref{main line} answers this question with an emphatic no.

For $k\ge 4$, Chudnovsky, Scott, and Seymour~\cite{chudnovsky2016induced} proved that intersection graphs of rectangles or segments in the plane with no induced cycle of length $k$ are $\chi$-bounded. Compared to this, Theorems~\ref{main} and~\ref{main line} show that the situation in $\mathbb{R}^3$ is much more dire, and that in a sense these graphs are much further away from being $\chi$-bounded.
For more on \mbox{$\chi$-boundedness}, see the recent survey by Scott and Seymour~\cite{scott2018survey}.

Tomon and Zakharov~\cite{tomon2020tur} recently showed that there exist bipartite box intersection graphs with girth at least 6 that have a super-linear number of edges. This resolved another problem of Kostochka~\cite{kostochka2004coloring}.
We extend this showing that even bipartite box intersection graphs with arbitrarily large girth can have a super-linear number of edges.

\begin{corollary}\label{corollary}
	There are bipartite intersection graphs of axis-aligned boxes in $\mathbb{R}^3$ with arbitrarily large girth and minimum degree.
\end{corollary}

As one may suspect, it is possible to prove Corollary~\ref{corollary} by modifying the construction we use to prove Theorem~\ref{main}. We leave this as an exercise, and instead note that it can also be instantly deduced as a corollary of Theorem~\ref{main} and a recent theorem of Kwan, Letzter, Sudakov, and Tran~\cite{kwan2020dence} that every triangle-free graph with minimum degree at least $c$ contains an induced bipartite subgraph of minimum degree at least $\Omega(\ln c / \ln \ln c)$.
We may also obtain an analogue of Corollary~\ref{corollary} for lines in $\mathbb{R}^3$ in the same way.

One pleasing consequence of the connection to arithmetic Ramsey theory is that we may obtain Burling's classical result as an application of Van der Waerden's theorem~\cite{vanderwaerden} on arithmetic progressions. To prove Theorems~\ref{main} and~\ref{main line} we require a certain strengthening of Van der Waerden's theorem which we introduce in the next section.

\section{Preliminaries}\label{section2}

A \emph{homothetic map} $f:\mathbb{R}^d \to \mathbb{R}^d$ is one of the form $f(x)=sx+\bar{c}$ for some $s\in \mathbb{R}_{>0}$ and $\bar{c}\in \mathbb{R}^d$. 
In other words, a homothetic map is a composition of uniform scaling and a translation.
A set $T'\subset \mathbb{R}^d$ is a \emph{homothetic copy} of a set $T\subset \mathbb{R}^d$ if there is a homothetic map $f$ such that $f(T)=T'$.
We say that a collection of distinct sets $T_1,\dots , T_k$ form a cycle $C$ of length $k$ if there exist distinct elements $x_1,\dots , x_k$ such that for all $i\in \{1,\dots ,k-1\}$ we have $x_i \in T_i \cap T_{i+1}$ and $x_k \in T_k \cap T_1$.
A triangle-free intersection graph of a collection of sets $\cal T$ has a cycle of length $k$ if and only if there is a collection of $k$ sets contained in $\cal T$ that form a cycle of length $k$.
We often find it convenient to consider colourings and cycles of objects in $\mathbb{R}^3$ directly rather than in their intersection graphs.

We require a sparse version of Gallai's theorem~\cite{rado1933studien} proven by Pr{\"o}mel and Voigt~\cite{promel1990sparse}. The one dimensional version is sufficient for our purposes.

\begin{theorem}[Pr{\"o}mel and Voigt~\cite{promel1990sparse}]\label{Gallai}
	Let $T\subset \mathbb{R}$ be a finite subset containing at least three elements and let $g$ and $k$ be positive integers. Then there exists a finite set $X\subset \mathbb{R}$ such that every $k$-colouring of $X$ contains a monochromatic homothetic copy $T'$ of $T$ and no set of at most $g-1$ homothetic copies of $T$ form a cycle.
\end{theorem}

Pr{\"o}mel and Voigt~\cite{promel1988sparse} also proved sparse versions of other theorems including the Hales--Jewett theorem~\cite{hales1963regularity} and the Graham--Rothschild theorem~\cite{graham1971ramsey}. A special case of Theorem~\ref{Gallai} is also equivalent to a sparse Van der Waerden's theorem.

The graphs we shall construct are based on a hypergraph variation of Tutte's~\cite{descartes1947three,descartes1954solution} construction of triangle-free graphs with arbitrarily large chromatic number that has been considered several times before~\cite{kostochka1999properties,toft1974color,nevsetvril1979short}. Although we need not make any mention of hypergraphs, their use is implicit in the application of Theorem~\ref{Gallai}, indeed this theorem essentially provides a very specialised hypergraph.

A necessarily more relaxed variation of Tutte's construction was also used to show that there are axis-aligned boxes in $\mathbb{R}^3$ with disjoint interiors whose intersection graph has arbitrarily large chromatic number~\cite{reed2008painting}.
Magnant and Martin~\cite{Magnant2011box} further showed that there are such intersection graphs where the boxes only intersect on their top and bottom faces (and thus are triangle-free), and have arbitrarily large chromatic number.

We remark that while Tutte's construction can be sparsified, these variations used for boxes with disjoint interiors cannot be;
an intersection graph of $n$ axis-aligned boxes in $\mathbb{R}^3$ with disjoint interiors, and with girth at least 5, has at most $24n$ edges.
To show this, first observe that if two boxes with disjoint interiors intersect, then there is a plane that contains a pair of intersecting faces, one from each of the two boxes. Glebov~\cite{glebov2002box} proved that every intersection graph of $r$ axis-aligned rectangles in the plane with girth at least 5 has at most $4r$ edges. Each of the 6 faces of a box is contained in a unique plane, so we may obtain the bound of $24n$ by applying Glebov's result to each plane contained in $\mathbb{R}^3$ and the rectangular faces that they contain.

It would be interesting to obtain improved bounds for both the number of edges and the chromatic number of these graphs.

\section{Boxes}\label{section3}

Before proving Theorem~\ref{main}, we need to introduce a more restricted setting to facilitate the inductive argument. 
Let $p_x:\mathbb{R}^3 \mapsto \mathbb{R}$ denote the projection onto the $x$-axis. Similarly for $p_y$ and $p_z$.
We say that an axis-aligned box $B\subseteq \mathbb{R}^3$ is a \emph{square box} if its top and bottom faces are squares.
Let $P_{x=y}\subset \mathbb{R}^3$ be the $x=y$ plane. We say that a square box $B\subset \mathbb{R}^3$ is \emph{grounded} if $B$ intersects $P_{x=y}$, and $p_x(a)\ge p_y(a)$ for all $a\in B$.
Notice that if $B$ is grounded, then $B \cap P_{x=y}$ is a vertical line segment, and an edge of the box.
A graph is a \emph{grounded square box} graph if it is the intersection graph of a collection of grounded square boxes.

We prove Theorem~\ref{main} for the subclass of grounded square box graphs. While the grounded condition is used to ease the inductive argument, achieving square boxes rather than general grounded axis-aligned boxes is simply a bonus we get for free from the proof.

\begin{theorem}\label{grounded boxes}
	For every $g\ge 3$ and $k\ge 1$, there exists a grounded square box graph with girth at least $g$ and chromatic number at least $k$.
\end{theorem}

\begin{proof}
	This is trivially true for $k\le 3$, as all odd cycles are grounded square box graphs. So we fix $g$ and proceed inductively on $k$. Let ${\cal B}_{g,k}$ be a collection of grounded square boxes whose intersection graph has girth at least $g$ and chromatic number at least $k$. We may assume that each box has a non-empty interior, and by possibly performing small uniform scaling and translations along $P_{x=y}$ to individual boxes, we may assume that the points $p_x(P_{x=y}\cap B)$ and $p_x(P_{x=y}\cap B')$ are distinct for each pair of distinct boxes $B, B'\in {\cal B}_{g,k}$.
	
	Let $T=p_{x}(P_{x=y}\cap {\cal B}_{g,k})\subset \mathbb{R}$.
	By Theorem~\ref{Gallai} there exists a finite set $X\subset \mathbb{R}$ such  that every $k$-colouring of $X$ contains a monochromatic homothetic copy $T'$ of $T$ and no set of at most $\lfloor \frac{g}{3} \rfloor$ homothetic copies of $T$ contained in $X$ form a cycle. Let $\cal T$ be the set of homothetic copies of $T$ in $X$.
	
	Choose some $0<\epsilon < \min\{|x-x'| : x,x' \in X, x \not= x'\}$. Now for each $x\in X$, let $B_x$ be the grounded square box $[x,x+\epsilon] \times [x-\epsilon,x] \times [0,1]$. Let ${\cal B}_{X}$ be the collection of all such grounded square boxes. Clearly no pair of these boxes intersect.
	
	For each $T'\in {\cal T}$:
	\begin{enumerate}
		\item [1)] let $I_{T'}$ be an interval contained in $[0,1]$, such that the intervals $\{I_{T'}: T'\in {\cal T}\}$ are pairwise disjoint,
		
		\item [2)] let $s_{T'}$ be some homothetic map that maps the intervals of $p_z({\cal B}_{g,k})$ to intervals contained in $I_{T'}$,
		
		\item [3)] let $f_{T'}$ be the homothetic map that maps $T$ to $T'$ in $\mathbb{R}$,
		
		\item [4)] let $h_{T'}: \mathbb{R}^3 \mapsto \mathbb{R}^3$ be the map such that $h_{T'}(x,y,z)=(f_{T'}(x),f_{T'}(y),s_{T'}(z))$ for all $(x,y,z) \in \mathbb{R}^3$, and lastly
		
		\item [5)] let ${\cal B}_{T'}=h_{T'}({\cal B}_{g,k})$.
	\end{enumerate}
	As $f_{T'}$ and $s_{T'}$ are both homothetic maps, ${\cal B}_{T'}$ is a collection grounded square boxes whose intersection graph is isomorphic to that of ${\cal B}_{g,k}$. Furthermore each box of ${\cal B}_{T'}$ intersects exactly one of the boxes $\{B_t: t\in T'\} \subset {\cal B}_X$, and no box of ${\cal B}_{T'}$ intersects a box of ${\cal B}_{T^*}$ for any $T^*\in {\cal T}\backslash \{T'\}$.
	
	Let ${\cal B}_{g,k+1}$ be the union of the grounded square boxes ${\cal B}_X$ and $({\cal B}_{T'} : T'\in {\cal T})$. It remains to show that the intersection graph of ${\cal B}_{g,k+1}$ has girth at least $g$ and chromatic number at least $k+1$.
	
	We handle the chromatic number first. Suppose for sake of contradiction that there is a $k$-colouring of ${\cal B}_{g,k+1}$. Then by choice of $X$, there is a $T'\in {\cal T}$ such that the boxes $\{B_t: t\in T'\}$ are monochromatic. Every box of ${\cal B}_{T'}$ intersects a box of $\{B_t: t\in T'\}$, so the boxes of ${\cal B}_{T'}$ must be $(k-1)$-coloured. But the intersection graph of ${\cal B}_{T'}$ has chromatic number at least $k$, a contradiction.
	
	Next we consider the girth. By the inductive assumption, a cycle $C$ of length less than $g$ in ${\cal B}_{g,k+1}$ must contain at least one box $B$ of ${\cal B}_X$. Let $B_1,B_2$ be the boxes immediately after and before $B$ in the cycle $C$. Then there exist distinct $T_1,T_2\in {\cal T}$ such that $B_1\in {\cal B}_{T_1}$ and $B_2\in {\cal B}_{T_2}$. But then as $\cal T$ contains no cycle of length less than $\lceil \frac{g}{3} \rceil$, we see that $C$ must contain at least $\lceil \frac{g}{3} \rceil$ boxes of ${\cal B}_X$. Furthermore the boxes ${\cal B}_X$ are pairwise disjoint and there is no box that intersects two boxes of ${\cal B}_X$. Hence all cycles must have length at least $3 \lceil \frac{g}{3} \rceil \ge g$ as required.
\end{proof}

It only takes a minor modification of this proof to obtain a proof of Burling's~\cite{burling1965coloring} classical result as an application of Van der Waerden's theorem. In fact, here we can already achieve a girth of 6. One just has to arrange at the start that in addition to the points $p_x(P_{x=y}\cap B)$ and $p_x(P_{x=y}\cap B')$ being distinct for each pair of distinct boxes $B, B'\in {\cal B}_{6,k}$, we also have that $p_x(P_{x=y}\cap B)\in \mathbb{Q}$ for each $B \in {\cal B}_{6,k}$. Then there is some arithmetic progression in $\mathbb{Q}$ that contains $p_x(P_{x=y}\cap {\cal B}_{6,k})$, and so we may use Van der Waerden's theorem in place of Theorem~\ref{Gallai}. For larger girth we could also choose to use the sparse Van der Waerden's theorem~\cite{promel1988sparse} in place of Theorem~\ref{Gallai}.

Another interesting proof of Burling's result was given by Krawczyk and Walczak~\cite{online}. Their proof uses an online colouring approach and reproduces the same graphs constructed by Burling.

\section{Lines}\label{section4}

Before proving Theorem~\ref{main line}, we first present Norin's~\cite{norin} construction of double shift graphs as intersection graphs of lines in $\mathbb{R}^3$.
For a positive integer $n$ and a finite set $S \subset \mathbb{R}$ with $|S|=n$, the double shift graph $G_{n}$ is isomorphic to the graph where the vertex set is all the triples $(a,b,c)$ with $a,b,c\in S$ and $a<b<c$, and where two vertices $(a,b,c)$ and $(d,e,f)$ are adjacent when $b=d$ and $c=e$.

For a positive integer $n$, the double shift graph $G_n$ can be constructed as an intersection graph of lines in $\mathbb{R}^3$ as follows.
Let $S\subseteq \mathbb{R}$ be an algebraically independent set with $|S|=n$. Then for each triple $(a,b,c)$ with $a,b,c\in S$ and $a<b<c$, let
\[L_{(a,b,c)}(t)= (ab + bc + t, abc + bt, ab^2c + (ab + bc)t).\]
Then by design, we have that $L_{(a,b,c)}(cd) = L_{(b,c,d)}(ab)$, and using the fact that $S$ is algebraically independent, it can be shown that all the points of intersection between these lines are of this form. Thus the intersection graph of these lines is isomorphic to the double shift graph $G_n$.
Erd{\H{o}}s and Hajnal~\cite{erdos1964shiftgraphs} showed that $\lim_{n \to \infty} \chi(G_n) = \infty$, and so intersection graphs of lines in $\mathbb{R}^3$ are not $\chi$-bounded.

Another related and more general class that this applies to is the intersection graphs of unit length straight line segments in $\mathbb{R}^3$. Unlike segments of arbitrary length, in the plane these graphs are $\chi$-bounded~\cite{suk2014mono}.

Next we prove Theorem~\ref{main line}, that there are intersection graphs of lines in $\mathbb{R}^3$ with arbitrarily large girth and chromatic number.
The proof is similar to that of Theorem~\ref{grounded boxes}, we just have to modify the geometric arguments for this setting.
We restate Theorem~\ref{main line} slightly differently for convenience.

\begin{theorem}\label{line}
	For every $g\ge 3$ and $k\ge 1$, there exist lines in $\mathbb{R}^3$ whose intersection graph has girth at least $g$ and chromatic number at least $k$.
\end{theorem}

\begin{proof}
	The theorem is trivially true for $k\le 3$, as all odd cycles are intersection graphs of lines in $\mathbb{R}^3$. So we fix $g$ and proceed inductively on $k$.
	Let ${\cal S}_{g,k}$ be a collection of lines in $\mathbb{R}^3$ whose intersection graph has girth at least $g$ and chromatic number at least $k$.
	
	Choose a plane $P\subset \mathbb{R}^3$ such that:
	\begin{itemize}
		\item every line of ${\cal S}_{g,k}$ intersects $P$, and
		
		\item no pair of distinct lines of ${\cal S}_{g,k}$ intersect $P$ at the same point.
	\end{itemize}
	
	Let $D$ be the set of points of $P$ contained in a line of ${\cal S}_{g,k}$, in particular there is a one to one correspondence between points of $D$ and lines in ${\cal S}_{g,k}$. 
	
	Now choose an auxiliary line $L$ contained in $P$ such that:
	\begin{itemize}
		\item no pair of distinct points in $D$ are contained in a single straight line segment $L'\subset P$ that is perpendicular to $L$, and
		
		\item there is no plane $Q$ such that $P\cap Q$ is a line perpendicular to $L$, the plane $Q$ contains a line $S$ of ${\cal S}_{g,k}$, and some translation of $Q$ contains a line $S^*$ of ${\cal S}_{g,k}$ that is not parallel to $S$.
	\end{itemize}
	
	Such a line $L$ can be chosen, since for non-parallel lines $S,S^*\in {\cal S}_{g,k}$, there is a unique plane $Q$ such that $Q$ contains $S$ and some translation of $Q$ contains $S^*$.
	
	For each $d\in D$, let $t_d$ be the point of $L$ such that there is a straight line segment contained in $P$ that is perpendicular to $L$ and contains both $t_d$ and $d$.
	Now let ${T=\{t_d : d\in D \}}$. By choice of $L$, there is a one to one correspondence between points of $T$ and lines of ${\cal S}_{g,k}$.
	
	By Theorem~\ref{Gallai}, there exists a finite set $X\subset L$ such  that every $k$-colouring of $X$ contains a monochromatic homothetic copy $T'$ of $T$ and no set of at most $\lfloor \frac{g}{3} \rfloor$ homothetic copies of $T$ contained in $X$ form a cycle. Let $\cal T$ be the set of homothetic copies of $T$ in $X$. For each $x\in X$, let $S_x$ be the line contained in $P$ that is perpendicular to $L$ and contains the point $x$. Let ${\cal S}_X = \{S_x : x\in X\}$. There is a one to one correspondence between lines of ${\cal S}_{g,k}$ and lines of $\{S_t : t\in T\}$. Notice that each line $S$ of ${\cal S}_{g,k}$ intersects its corresponding line $S_t$ and no other line contained in ${\cal S}_X$.
	
	Let ${\cal S}_{T}= {\cal S}_{g,k}$.
	For each $T'\in {\cal T}\backslash \{T\}$, let ${\cal S}_{T'}$ be a homothetic copy of ${\cal S}_{T}$ with a homothetic map $f_{T'}$ such that:
	\begin{itemize}
		\item each line of ${\cal S}_{T'}$ intersects its corresponding line of $\{S_{t'} : t' \in T'\}$ and no other line of ${\cal S}_X$, and
		
		\item for distinct $T_1,T_2\in {\cal T}$, no line of ${\cal S}_{T_1}$ intersects a line of ${\cal S}_{T_2}$.
	\end{itemize}
	
	For each $T'\in {\cal T}\backslash {T}$, the homethetic map $f_{T'}$ would be a composition of a homethetic map that maps $T$ to $T'$ in $\mathbb{R}^3$ (ensuring the correct intersections between lines of ${\cal S}_{T'}$ and $\{S_{t'} : t' \in T'\}$) and some translation in the direction parallel to the lines of ${\cal S}_X$ so as to avoid lines of ${\cal S}_{T'}$ intersecting lines of another homothetic copy of ${\cal S}_{T}$.
	Such translations exist by the choice of $L$, since for each $x\in X$ and translated copies $S, S^*$ of lines in ${\cal S}_{T}= {\cal S}_{g,k}$ such that $S_x, S , S^*$ are coplanar, the lines $S$ and $S^*$ must be parallel.
	
	Let ${\cal S}_{g,k+1}$ be the union of the lines of ${\cal S}_X$ and $({\cal S}_T : T\in {\cal T})$.
	The proof that the intersection graph of ${\cal S}_{g,k+1}$ has girth at least $g$ and chromatic number at least $k+1$ is essentially the same as for the boxes in Theorem~\ref{grounded boxes}.
\end{proof}

\section*{Acknowledgments} 
The author would like to thank Jim Geelen, Matthew Kroeker, and Rose McCarty for helpful discussions, which in particular improved the proof of Theorem~\ref{main}.
The author would also like to thank Tom Johnston for pointing out an issue in the proof of Theorem~\ref{line} in a previous version of this manuscript.
The author thanks Sergey Norin for allowing the inclusion of his construction of double shift graphs as line intersection graphs. Lastly the author thanks the anonymous referees for some helpful suggestions on the presentation of the paper.

\bibliographystyle{amsplain}


\begin{aicauthors}
\begin{authorinfo}[jd]
  James Davies\\
  Department of Combinatorics and Optimization\\
  University of Waterloo\\
  Waterloo, Canada\\
  jgdavies\imageat{}uwaterloo\imagedot{}ca 
\end{authorinfo}

\end{aicauthors}

\end{document}